\documentclass[english, 10pt]{amsart}

\usepackage{babel}
\usepackage{amstext}
\usepackage{amsmath}
\usepackage{amssymb}
\usepackage{amsfonts}
\usepackage{latexsym}
\usepackage{ifthen}
\usepackage[all,2cell,dvips]{xy} \UseAllTwocells \SilentMatrices
\xyoption{all}
\usepackage{verbatim}

\newcommand{\rk}{{\rm rk}}

\newtheorem{lemma1}{}[section]

\newenvironment{lemma}{\begin{lemma1}{\bf Lemma.}}{\end{lemma1}}

\newenvironment{theorem}{\begin{lemma1}{\bf Theorem.}}{\end{lemma1}}
\newenvironment{proposition}{\begin{lemma1}{\bf Proposition.}}{\end{lemma1}}
\newenvironment{corollary}{\begin{lemma1}{\bf Corollary.}}{\end{lemma1}}
\newenvironment{remark}{\begin{lemma1}{\bf Remark.}\rm}{\end{lemma1}}
\newenvironment{definition}{\begin{lemma1}{\bf Definition.}}{\end{lemma1}}

\newenvironment{setup}{\begin{lemma1}{\bf Setup.} \rm}{\end{lemma1}}

\newenvironment{remark*}{{\bf Remark.}}{}
\newenvironment{example*}{{\bf Example.}}{}

\newcommand{\R}{\ensuremath{\mathbb{R}}}
\newcommand{\Q}{\ensuremath{\mathbb{Q}}}

\newcommand{\N}{\ensuremath{\mathbb{N}}}
\newcommand{\PP}{\ensuremath{\mathbb{P}}}

\newcommand{\holom}[3]{\ensuremath{#1:#2  \rightarrow #3}}
\newcommand{\fibre}[2]{\ensuremath{#1^{-1} (#2)}}

\makeatletter
\ifnum\@ptsize=0  \fi \ifnum\@ptsize=2  \fi 

\newcommand\sF{{\mathcal F}}
\newcommand\sG{{\mathcal G}}

\newcommand\sO{{\mathcal O}}

\newcommand{\chow}[1]{\ensuremath{\mathcal{C}(#1)}}

\newcommand{\Chow}[1]{\ensuremath{\mathcal{C}(#1)}}

\title{Twisted cotangent sheaves and a Kobayashi-Ochiai theorem for foliations} 
\date{July 12, 2013}

\author{Andreas H\"oring}

\subjclass[2010]{14F10, 37F75, 14M22, 14E30, 14J40}
\keywords{cotangent sheaf, foliations, Kobayashi-Ochiai theorem}
\thanks{}

\address{Andreas H\"oring, Universit{\'e} Pierre et Marie Curie, Institut de math{\'e}matiques de Jussieu,
Projet Topologie et g{\'e}om{\'e}trie alg{\'e}briques, Case 247,  4 place Jussieu, 75005 Paris, France}
\email{hoering@math.jussieu.fr}

\begin{document}

\begin{abstract}
Let $X$ be a normal projective variety, and let $A$ be an ample Cartier divisor on $X$.
We prove that the twisted cotangent sheaf $\Omega_X \otimes A$ is generically nef with respect to
the polarisation  $A$
unless $X$ is a projective space. As an application we prove a Kobayashi-Ochiai theorem
for foliations: if $\sF \subsetneq T_X$ is a foliation of rank $r$
such that $\det \sF \equiv i_{\sF} A$, then we have $i_{\sF} \leq r$.
\end{abstract}

\maketitle


\section{Introduction}

Let $X$ be a normal projective variety of dimension $n$, and let $A$ be an ample Cartier divisor on $X$. A classical theorem
of Kobayashi and Ochiai \cite{KO73} characterises the projective space as the unique variety such
that $K_X + (n+1) A$ is trivial and hyperquadrics as the varieties such that $K_X + n A$ is trivial.
This result can be seen as the starting point of the adjunction theory of projective manifolds, combined
with the minimal model program this theory gives us very precise information about the relation between
the positivity of the canonical divisor $K_X$ and some polarisation $A$ on $X$. The aim of this paper is to prove a basic result
relating the positivity of the cotangent sheaf $\Omega_X$ to a polarisation:

\begin{theorem} \label{theoremmain} 
Let $X$ be a normal projective variety of dimension $n$, and let $A$ be an ample Cartier divisor on $X$.
Then one of the following holds:
\begin{enumerate}
\item We have $(X,  \sO_X(A)) \simeq (\PP^n, \sO_{\PP^n}(1))$; or
\item the twisted cotangent sheaf $\Omega_X \otimes A$ is generically nef with respect to $A$.
\end{enumerate}
If the  twisted cotangent sheaf $\Omega_X \otimes A$ is not generically ample with respect to 
$A$, one of the following holds:
\begin{enumerate}
\item There exists a normal projective variety $Y$ of dimension at most $n-1$
and a vector bundle $V$ on $Y$ such that
$X' := \PP(V)$
admits a birational morphism \holom{\mu}{X'}{X} such that $\sO_{X'}(\mu^* A) \simeq \sO_{\PP(V)}(1)$; or
\item $(X,  \sO_X(A)) \simeq (Q^n, \sO_{Q^n}(1))$ where $Q^n \subset \PP^{n+1}$ is a (not necessarily smooth) quadric hypersurface.
\end{enumerate}
\end{theorem}

As an application we obtain a bound for the index of a $\Q$-Fano distribution:

\begin{corollary} \label{corollaryKO}
Let $X$ be a normal projective variety, and let $A$ be an ample Cartier divisor on $X$.
Let $\sF \subsetneq T_X$ be a distribution of rank $r$ such that $\det \sF$ is $\Q$-Cartier 
and $\det \sF \equiv i_{\sF} A$.
Then we have
$
i_{\sF} \leq r.
$ 
\end{corollary}

Indeed if we had $i_{\sF} > r$, then 
$
\Omega_X \otimes A \rightarrow \sF^* \otimes A
$
would be a quotient with antiample determinant, in particular $\Omega_X \otimes A$
would not be generically nef with respect to $A$,
in contradiction to the first part of Theorem \ref{theoremmain}.
This Kobayashi-Ochiai theorem for foliations generalises similar results obtained recently
by Araujo and Druel \cite[Thm.1.1, Sect.4]{AD12}. 
Note that the method of proof is quite different: while the work of Araujo and Druel is based
on the geometry of the general log-leaf, Theorem \ref{theoremmain}
improves a semipositivity result for $\Omega_X \otimes A$ proven in \cite{a11}.
The proof of this semipositivity results relies on comparing the positivity of a foliation $\sF \otimes A$
along a very general curve $C \subset X$ with the positivity of a relative canonical divisor $K_{X'/Y}+r \mu^* A$
(cf. Section \ref{sectiontwistedcotangent}). The advantage of this technique is that we can make weaker assumptions on the variety $X$ or the foliation $\sF$, the disadvantage is that we do not get any information
about the singularities of the foliation $\sF$. For the description of the boundary case in Corollary \ref{corollaryKO} we therefore follow closely the arguments in  \cite[Thm.4.11]{AD12}:

\begin{theorem} \label{theoremKO}  
Let $X$ be a normal projective variety, and let $A$ be an ample Cartier divisor on $X$.
Let $\sF \subsetneq T_X$ be a foliation of rank $r$ such
that $\det \sF$ is $\Q$-Cartier and $\det \sF \sim_\Q r A$. 

Then $X$ is a generalised cone, more precisely there exists a normal projective variety $Y$
and an ample line bundle $M$ on $Y$ such that $X' := \PP(M \oplus \sO_Y^{\oplus r})$
admits a birational morphism \holom{\mu}{X'}{X} such that $\mu_* T_{X'/Y} = \sF$
and $\sO_{X'}(\mu^* A) \simeq \sO_{\PP(M \oplus \sO_Y^{\oplus r})}(1)$.
\end{theorem}

This statement generalises a classical theorem of Wahl \cite{Wah83, Dru04} on ample line bundles
contained in the tangent sheaf. If $X$ is smooth, then Corollary \ref{corollaryKO} and Theorem \ref{theoremKO}
are special cases of the characterisation of the projective space and hyperquadrics by Araujo, Druel and Kov\'acs \cite[Thm.1.1]{ADK08}.
Vice versa, Theorem \ref{theoremmain} yields a weak version of \cite[Thm.1.2]{ADK08}, \cite[Thm.1.1]{Par10} for normal varieties:

\begin{corollary}
Let $X$ be a normal projective variety of dimension $n$, and let $A$ be an ample Cartier divisor on $X$.
Suppose that for some positive $m \in \N$ we have
$$
H^0(X, (T_X \otimes A^*)^{[\otimes m]}) \neq 0,
$$
where $(T_X \otimes A^*)^{[\otimes m]}$ is the bidual of $(T_X \otimes A^*)^{\otimes m}$.
Then one of the following holds:
\begin{enumerate}
\item There exists a normal projective variety $Y$ of dimension at most $n-1$
and a vector bundle $V$ on $Y$ such that
$X' := \PP(V)$
admits a birational morphism \holom{\mu}{X'}{X} such that $\sO_{X'}(\mu^* A) \simeq \sO_{\PP(V)}(1)$; or
\item $(X,  \sO_X(A)) \simeq (Q^n, \sO_{Q^n}(1))$ where $Q^n \subset \PP^{n+1}$ is a (not necessarily smooth) quadric hypersurface.
\end{enumerate}
\end{corollary}

Indeed if $\Omega_X \otimes A$ is generically ample, then $(\Omega_X \otimes A)^{\otimes m}$ is generically ample
for every positive $m \in \N$ by
\cite[Cor.6.1.16]{Laz04}. In particular its dual does not have any non-zero global section, in contradiction to the assumption.
Thus the second part of Theorem \ref{theoremmain} applies.

{\bf Acknowledgements.} I would like to thank Fr\'ed\'eric Han for patiently answering
my questions about Schur functors. The author was partially supported by the A.N.R. project ``CLASS''.

\section{Notation} \label{sectionnotation}

We work over the complex numbers, topological notions always refer to the Zariski topology.
For general definitions we refer to \cite{Har77} and \cite{Laz04}.
We will frequently use the terminology and results 
of the minimal model program (MMP) as explained in \cite{KM98} or \cite{Deb01}.
For some standard definitions concerning the adjunction theory of (quasi-)polarised varieties we
refer to \cite{Fuj89, BS95, a12}.

\begin{definition} \label{definitionmrgeneral}
Let $X$ be a normal projective variety of dimension $n$, and let $H_\bullet$ be a multipolarisation, that is
$H_\bullet = \{ H_1, \ldots, H_{n-1} \}$ is a collection of $n-1$ ample Cartier divisors.
A MR-general curve $C \subset X$ is an intersection
\[
D_1 \cap \ldots \cap D_{n-1}
\]
for general $D_j \in | m_j H_j |$ where $m_j \gg 0$. 
\end{definition}

The abbreviation MR stands for Mehta-Ramanathan, alluding to the well-known fact \cite{MR82}
that the Harder-Narasimhan filtration of a torsion-free sheaf commutes with restriction to a MR-general curve. 

Let $X$ be a normal projective variety, and let  $\sF$ be a coherent sheaf that is locally free in codimension one.
A MR-general curve $C$ is contained in the open set 
where $X$ is smooth and $\sF$ is locally free, in particular  $\sF|_C$ is a vector bundle.

\begin{definition} \label{definitiongenericallynef}
Let $X$ be a normal projective variety of dimension $n$, and let 
$\sF$   be a coherent sheaf on $X$ that is locally free in codimension one.
The sheaf $\sF$ 
is generically nef (resp. ample) with respect to a multipolarisation $H_\bullet$
if its restriction to a MR-general curve $C$ is a nef (resp. ample) vector bundle.
\end{definition}

\begin{remark} \label{remarkmrgeneral}
Since the Harder-Narasimhan filtration of a torsion free sheaf $\sF$ commutes with restriction to $C$,
the sheaf $\sF$ is generically nef with respect to $H_\bullet$ if and only if
$$
\mu_{H_\bullet}(\sF_k/\sF_{k-1}) \geq 0,
$$
where $\sF_k/\sF_{k-1}$ is the last graded piece of the Harder-Narasimhan filtration and $\mu_{H_\bullet}(.)$ the slope with respect to $H_\bullet$.
\end{remark}

The following definition is a slight modification of the definition of a generalised cone \cite[1.1.8]{BS95}: 

\begin{definition} \label{definitionconestructure}
Let $F$ be a normal projective variety, and let $A$ be an ample Cartier divisor on $F$. We say that $F$, or more precisely the polarised variety $(F, \sO_F(A))$,  has a cone structure if there exists a normal projective variety $G$,
an ample vector bundle $M$ on $G$, and a positive $d \in \N$ such that
the projectivised vector bundle $F' := \PP(M \oplus \sO_{G}^{\oplus d})$ admits
a birational morphism  $\holom{\mu}{\PP(M \oplus \sO_{G}^{\oplus d})}{F}$ such that $\sO_{F'}(\mu^* A) \simeq \sO_{\PP(M \oplus \sO_{G}^{\oplus d})}(1)$.
\end{definition}

Cone structures appear naturally in the classification of normal projective varieties such that $K_X$
is not necessarily $\Q$-Cartier:

\begin{lemma} \label{lemmaadjunction}
Let $F$ be a normal projective variety of dimension $r$ that is rationally connected, and let $A$ be an ample Cartier divisor on $F$.
Let \holom{\sigma}{\tilde F}{F} be a desingularisation. Suppose that we have
$$
\kappa(\tilde F, K_{\tilde F}+j \sigma^* A)= - \infty
$$
for every $j \in \Q$ such that $0 \leq j <r$. Then the polarised variety $(F, \sO_F(A))$ is isomorphic 
to one of the following varieties:
\begin{enumerate}
\item $(\PP^r, \sO_{\PP^r}(1))$; or 
\item $(Q^r, \sO_{Q^r}(1))$ where $Q^r \subset \PP^{r+1}$ is a normal quadric hypersurface; or
\item $(\PP(V), \sO_{\PP(V)}(1))$ where $V$ is an ample vector bundle on $\PP^1$; or
\item $F$ has a cone structure given by some nef and big vector bundle on $\PP^1$.
\end{enumerate}
If we have $\kappa(\tilde F, K_{\tilde F}+r \sigma^* A)= - \infty$,
then $(F, \sO_F(A))$ is isomorphic  to $(\PP^r, \sO_{\PP^r}(1))$.
\end{lemma}

\begin{proof}
By the nonvanishing theorem \cite[Thm.D.]{BCHM10} we know that 
$K_{\tilde F}+j \sigma^* A$ is not pseudoeffective for every $0 \leq j<r$.
Using the terminology of \cite{a12, And13} we will run a $K_{\tilde F}+(r-\varepsilon) \sigma^* A$-MMP 
$$
\tilde F:= \tilde F_0 \dashrightarrow \tilde F_1 \dashrightarrow \ldots \dashrightarrow \tilde F_s 
$$
where $0<\varepsilon \ll \frac{1}{2}$.
Its outcome is a quasi-polarised variety $(\tilde F_{s}, \sO_{\tilde F_{s}}(A_s))$ 
with terminal singularities admitting an elementary contraction of fibre type \holom{\psi}{\tilde F_s}{G}
such that
$K_{\tilde F_s}+(r-\varepsilon)A_s$ is $\psi$-antiample.
Let $\Gamma$ be the extremal ray contracted by the first step of this MMP. If the corresponding
elementary contraction is birational, we know by \cite[Prop.3.6]{And13} that $\sigma^* A \cdot \Gamma=0$.
Since $A$ is ample this implies that every fibre of the extremal contraction
is contained in a $\sigma$-fibre. Thus by the rigidity lemma there exists a morphism $\tilde F_1 \rightarrow F$.
Arguing inductively (cf. the proof of \cite[Prop.1.3]{a12}) we see that there exists
a birational morphism $\nu: \tilde F_s \rightarrow F$ such that $A_s \simeq \nu^* A$.
Note that $\nu$ is an isomorphism if and only if $A_s$ is ample.
 
By \cite[Prop.3.5]{And13}, \cite[Table 7.1]{BS95} we know that $(\tilde F_{s}, \sO_{\tilde F_{s}}(A_s))$ is one of the following quasi-polarised varieties: 
\begin{enumerate}
\item $(\PP^r, \sO_{\PP^r}(1))$; or
\item $(Q^r, \sO_{Q^r}(1))$ where $Q^r \subset \PP^{r+1}$ is a hyperquadric; or
\item a $(\PP^{r-1}, \sO_{\PP^{r-1}}(1))$ bundle over $\PP^1$.
\end{enumerate}
The first two cases correspond to the first two cases in the statement. In the third case 
we have $\tilde F_s \simeq \PP(V:=\psi_* \sO_{\tilde F_s}(A_s))$ and $V$ is nef and big.
If $V$ is not ample, $F$ has a cone structure, otherwise $\tilde F_s \simeq F$ is a projective bundle.
This proves the first statement, the second statement is an immediate consequence of the classification
obtained in the first part. 
\end{proof}

\begin{lemma} \label{lemmaquadricbundle}
Let $X_C$ be a normal projective variety of dimension $r+1$, and let $A$ be a Cartier divisor on $X_C$.
Let \holom{p_C}{X_C}{C} be a fibration onto a smooth curve $C$ such that $A$ is $p_C$-ample.
Suppose moreover that the general fibre $(F, \sO_F(A))$ is isomorphic to $(Q^r, \sO_{Q^r}(1))$ 
where $Q^r \subset \PP^{r+1}$ is a quadric hypersurface.

Then $X_C \rightarrow C$ is a quadric bundle, i.e. the variety $X_C$ has at most canonical singularities and there exists a Cartier divisor $M$ on $C$ such that
\begin{equation} \label{eqnquadricbundle}
K_{X_C/C}+r A \simeq p_C^* M.
\end{equation}
\end{lemma}

\begin{remark*} In the situation above
every fibre of the fibration $X_C \rightarrow C$ is isomorphic to a hyperquadric:
for integral fibres this is proven in \cite[Cor.5.5]{Fuj75}, for non-integral fibres the details are 
tedious and left to the interested reader.
\end{remark*}

\begin{proof} Note first that up to replacing $A$ by $A+p_C^* D$ for $D$ a sufficiently ample divisor
on $C$ we can suppose without loss of generality that $A$ is ample.
Let $\holom{\nu}{\tilde X_C}{X_C}$ be the canonical modification of $X_C$, that is $\tilde X_C$
is the unique normal projective variety with at most canonical singularities such that
$K_{\tilde X_C}$ is $\nu$-ample\footnote{The existence of the canonical modification is a consequence
of \cite{BCHM10}, cf. the forthcoming book \cite{Kol13}.}.
Since a normal quadric has at most canonical singularities, we see that $\nu$ is an isomorphism
over the generic point of $C$. Thus the restriction of $K_{\tilde X_C}+r \nu^* A$
to the general fibre of $p_C \circ \nu$ is trivial. 

Suppose first that $K_{\tilde X_C}+r \nu^* A$ is not 
$(p_C \circ \nu)$-nef. Then there exists a Mori contraction \holom{\psi}{\tilde X_C}{X_C'}
contracting an extremal ray $\Gamma$ such that $(K_{\tilde X_C}+r \nu^* A) \cdot \Gamma<0$.
Since $K_{\tilde X_C}+r \nu^* A$ is $(p_C \circ \nu)$-pseudoeffective, the contraction $\psi$ is birational.
Since $(K_{\tilde X_C}+r \nu^* A) \cdot \Gamma<0$ we know by \cite[Thm.2.1, II,i]{And95} 
that all the $\psi$-fibres have dimension strictly larger than $r$ unless $\nu^* A \cdot \Gamma=0$. 
Since $\dim X_C=r+1$ we see that $\nu^* A \cdot \Gamma=0$. The divisor
$A$ being ample this implies that the $\psi$-fibres are contained in the $\nu$-fibres. Yet $K_{\tilde X_C}$ is $\nu$-ample, a contradiction.

Thus $K_{\tilde X_C}+r \nu^* A$ is $(p_C \circ \nu)$-nef and trivial on the general fibre.
By a well-known application of Zariski's lemma \cite[Lemma 8.2]{BHPV04} this implies that
$$
K_{\tilde X_C}+r \nu^* A \simeq \nu^* p_C^* M
$$
for some Cartier divisor $M$ on $C$. In particular $K_{\tilde X_C}$ is $\nu$-trivial.
Yet $K_{\tilde X_C}$ is $\nu$-ample, so $\nu$ is an isomorphism.
\end{proof}

\section{The twisted cotangent sheaf} \label{sectiontwistedcotangent}

The setup of the proof of Theorem \ref{theoremmain} is analogous \cite[Thm.3.1]{a11}:

\begin{setup} \label{setup}
Let $X$ be a normal projective variety of dimension $n$, and let $A$ be an ample Cartier divisor on $X$.
Let $H_\bullet := (H_1, \ldots, H_{n-1})$ be a multipolarisation on $X$.
Denote by $T_X:=\Omega_X^*$ the tangent sheaf of $X$,
and let 
\[
0= \sF_0 \subsetneq \sF_1 \subsetneq \ldots \subsetneq \sF_k=T_X
\]
be the Harder-Narasimhan filtration of $T_X$ with respect to $H_\bullet$. 
Then for every $i \in \{ 1, \ldots, k \}$, the graded pieces $\sG_i:=\sF_i/\sF_{i-1}$ are semistable torsion-free sheaves
and if $\mu_{H_\bullet}(\sG_i)$ denotes the slope, we have a strictly decreasing sequence
\begin{equation} \label{eqnslopedecreases}
\mu_{H_\bullet}(\sG_1) > \mu_{H_\bullet}(\sG_2) > \ldots > \mu_{H_\bullet}(\sG_k).
\end{equation}
Since twisting with a Cartier divisor does not change the stability properties of a torsion-free
sheaf, the Harder-Narasimhan filtration of $T_X \otimes A^*$ is 
\[
0= \sF_0 \otimes A^* \subsetneq \sF_1\otimes A^* \subsetneq \ldots \subsetneq \sF_k \otimes A^*=T_X \otimes A^*
\]
with graded pieces $\sG_i \otimes A^*$ and slopes
\[
\mu_{H_\bullet}(\sG_i \otimes A^*) = \mu_{H_\bullet}(\sG_i)-A \cdot \prod_{i=1}^{n-1} H_i.
\]
Since stability is invariant under replacing the multipolarisation $H_\bullet$ by some positive multiple, we can suppose that 
the polarisations $H_i$ are very ample 
and $$
C :=D_1 \cap \ldots \cap D_{n-1}
$$
with $D_i \in |H_i|$ is a MR-general curve.

{\em Suppose now that $\Omega_X \otimes A$ is not generically ample with respect to
the multipolarisation $H_\bullet$, i.e. suppose that we have $\mu_{H_\bullet}(\sF_1 \otimes A^*) \geq 0$.} 

Fix a $l \in \{ 1, \ldots, k \}$ 
such that $\mu_{H_\bullet}(\sG_l \otimes A^*) \geq 0$. 
By the Mehta-Ramanathan theorem \cite[Thm.6.1]{MR82} the Harder-Narasimhan filtration commutes with restriction to $C$, 
so for $i \in \{ 1, \ldots, l \}$, the vector bundles $(\sG_i \otimes A^*)|_C$ are semistable with non-negative slope, hence 
nef. Since $A$ is ample, the vector bundles $(\sG_i)|_C$ are ample. Thus $\sF_l|_C$ is ample and 
we have 
\begin{equation} \label{eqnslopes}
\mu_{H_\bullet}(\sF_l \otimes A^*) = \sum_{i=1}^l \frac{\rk \sG_i}{\rk \sF_l} \mu_{H_\bullet}(\sG_i \otimes A^*) \geq 0. 
\end{equation}
Note that by \eqref{eqnslopedecreases} we have equality if and only if $l=1$ and $\mu_{H_\bullet}(\sF_1 \otimes A^*)=0$. 
We know by
standard arguments in stability theory \cite[p.61ff]{MP97} 
that $\sF_l$ is integrable, moreover the MR-general curve $C$ does not meet the singular
locus of the foliation.
Thus we can apply the
Bogomolov-McQuillan theorem \cite[Thm.0.1]{BM01}, \cite[Thm.1]{KST07} 
to see that the closure of a $\sF_l$-leaf
through a generic point of $C$ is algebraic and rationally connected. 
Since $C$ moves in a covering family the  general $\sF_l$-leaf is algebraic with rationally connected closure. 

We set $r:= \rk \sF_l$. If $\chow{X}$ denotes the Chow variety of $X$, we get a rational map $X \dashrightarrow \chow{X}$ that sends a general point $x \in X$ to the closure of the unique leaf through $x$.
 Let $Y$ be the normalisation of the closure of the image,
and let $X'$ be the normalisation of the universal family over $Y$. 
By construction the natural map \holom{\mu}{X'}{X} is birational and the fibration  $\holom{\varphi}{X'}{Y}$
is equidimensional of dimension $r$, the general fibre being the normalisation of the closure of a general $\sF_l$-leaf.
The restriction of $\mu$ to every $\varphi$-fibre is finite, so $\mu^* A$ is $\varphi$-ample.

Let $F$ be a general $\varphi$-fibre. The following lemma describes $F$:
\end{setup}

\begin{lemma} \label{lemmaidentifyfibres}
In the situation of Setup \ref{setup}, the polarised variety 
$(F, \sO_F(\mu^* A))$
is isomorphic to one of the following varieties:
\begin{enumerate}
\item $(\PP^r, \sO_{\PP^r}(1))$; or 
\item $(Q^r, \sO_{Q^r}(1))$ where $Q^r \subset \PP^{r+1}$ is a normal quadric hypersurface; or
\item $(\PP(V), \sO_{\PP(V)}(1))$ where $V$ is an ample vector bundle on $\PP^1$; or
\item $F$ has a cone structure over $\PP^1$ (cf. Definition \ref{definitionconestructure}).
\end{enumerate}
If $(F, \sO_F(\mu^* A)) \not\simeq (\PP^r, \sO_{\PP^r}(1))$, we have
$l=1$ and $\mu_{H_\bullet}(\sF_1 \otimes A^*)=0$.
\end{lemma}

\begin{proof}
Note that for every $0 < \varepsilon \ll 1$ the $\Q$-twisted\footnote{Cf. \cite[Ch.6.2]{Laz04} for the definition of $\Q$-twists.} 
cotangent sheaf $\Omega_X\hspace{-0.8ex}<\hspace{-0.8ex}(1-\varepsilon) A\hspace{-0.8ex}>$ is not generically nef.
If $\mu_{H_\bullet}(\sF_l \otimes A^*)>0$ the same holds for $\Omega_X\hspace{-0.8ex}<\hspace{-0.8ex}(1+\varepsilon) A\hspace{-0.8ex}>$.

Let $\sigma: \tilde X \rightarrow X'$ be a desingularisation, and $\tilde F$ a general fibre of the induced fibration
$\varphi \circ \sigma$. Then by \cite[Thm.3.1]{a11} (or rather its proof) we have
$$
\kappa(\tilde F, K_{\tilde F}+j (\sigma^* \mu^* H)|_{\tilde F})= - \infty
$$
for every $j \in \Q$ such that $0 \leq j <r$. 
Moreover if $\mu(\sF_l \otimes A^*)>0$ we also have $\kappa(\tilde F, K_{\tilde F}+r (\sigma^* \mu^* H)|_{\tilde F})= - \infty$.
The statement is now an immediate consequence of \eqref{eqnslopes} and Lemma \ref{lemmaadjunction}.
\end{proof}

The next propositions determine the structure of the fibre space $X' \rightarrow Y$:

\begin{proposition} \label{propositionlinearspaces}
Suppose that we are in the situation of Setup \ref{setup}.
\begin{enumerate}
\item Suppose that $(F, \sO_F(\mu^* A))$ is a linear projective space.
Then we have $X' \simeq \PP(V)$ where $V:= \varphi_* (\sO_{X'}(\mu^* A))$. 
\item Suppose that $(F, \sO_F(\mu^* A))$ is a $\PP^{r-1}$-bundle or has a cone structure.
Then there exists a normal projective variety $\tilde Y$ of dimension $n-r+1$
and a vector bundle $\tilde V$ on $\tilde Y$ such that
$\tilde X := \PP(\tilde V) \rightarrow \tilde Y$ admits a birational morphism \holom{\tilde \mu}{\tilde X}{X} 
such that $\sO_{\tilde X}(\tilde \mu^* A) \simeq \sO_{\PP(\tilde V)}(1)$.
\end{enumerate}
\end{proposition}

\begin{proof}
Statement a) is shown in \cite[Prop.4.10]{AD12} which improves \cite[Prop.3.5]{a18}.

For the proof of statement b) note that the $\PP^{r-1}$-bundle structure (resp. cone structure) 
on the general $\sF_l$-leaf defines an algebraically integrable foliation $\sF' \subset \sF_l$ of rank $r-1$.
As in Setup \ref{setup} we define $\tilde Y$ to be the normalisation of 
the closure of the locus in $\Chow{X}$ parametrising the general $\sF'$-leaf,
and $\tilde X$ as the normalisation of the universal family over $\tilde Y$. By construction 
the general fibre $\tilde F$ is isomorphic to $\PP^{r-1}$ and $\tilde \mu^* A$ restricted to  $\tilde F$
is the hyperplane divisor. Thus we can again apply \cite[Prop.4.10]{AD12}.
\end{proof}

\begin{proposition} \label{propositionquadric}
Suppose that we are in the situation of Setup \ref{setup}, and
suppose that $(F, \sO_F(\mu^* A))$ is a quadric.
Then $Y$ is a point, so $X$ itself is a quadric.
\end{proposition}

\begin{proof}
We will argue by contradiction and suppose that $Y$ has positive dimension.
Let $C \subset X$ be a MR-general curve, then $C$ does not meet the image of 
the $\mu$-exceptional locus so we can identify it to a curve in $X'$.
Denote by $X_C$ the normalisation of the fibre product $X' \times_Y C \subset X' \times C$,
and let \holom{p_{X'}}{X_C}{X'} be the natural map to the first factor.
The fibration $X' \times_Y C \rightarrow C$
admits a natural section 
$$
C \rightarrow X' \times_Y C \subset X' \times C, \ c \ \mapsto (c,c), 
$$ 
by the universal property of the normalisation we get a section of \holom{p_C}{X_C}{C}
which we denote by \holom{s}{C}{X_C}.
By \cite[Rem.19]{KST07} the normal variety $X_C$ is smooth in an analytic neighbourhood $U \subset X_C$ of $s(C)$ 
and 
\[
T_{X_C/C}|_U \simeq (p_{X'}^* \mu^* \sF_1)|_U. 
\]
Since $\mu_{H_\bullet}(\sF_1 \otimes A^*)=0$ by Lemma \ref{lemmaidentifyfibres} we have
\begin{equation} \label{eqnslopezero}
(K_{X_C/C} + r p_{X'}^* \mu^* A) \cdot s(C) = (K_{\sF_1} + r A) \cdot C = 0.
\end{equation}
By Lemma \ref{lemmaquadricbundle} the fibration $\holom{p_C}{X_C}{C}$ is a quadric bundle, i.e. the divisor 
$K_{X_C/C}$ is $\Q$-Cartier
and there exists a line bundle $M$ on $C$ such that 
$$
K_{X_C/C} + r p_{X'}^* \mu^* A \simeq p_C^* M.
$$ 
By \eqref{eqnslopezero} we have $p_C^* M \cdot s(C)=0$, so $M$ is numerically trivial.
This shows that $-K_{X_C/C}$ is nef and big. Yet this is impossible\footnote{Cf. \cite[Thm.5.1]{AD13} for a proof
in a more general setting.}: by Lemma \ref{lemmaquadricbundle}
the variety $X_C$ has at most canonical singularities and $p_C^* K_C - K_{X_C}$ is ample, 
so by the Kawamata-Viehweg vanishing theorem we have
$H^1(X_C, p_C^* K_C)=0.$
Yet by relative Kawamata-Viehweg vanishing the higher direct images $R^j (p_C)_* \sO_{X_C}$ 
vanish for $j \geq 1$ so we get 
$$
0 = H^1(C, K_C) \simeq H^0(C, \sO_C),
$$
a contradiction.
\end{proof}

\section{Proof of the main results}

We start by proving a statement in the case of an arbitrary multipolarisation:

\begin{theorem} \label{theoremmultipolarisation} 
Let $X$ be a normal projective variety of dimension $n$, and let $A$ be an ample Cartier divisor on $X$.
Let $H_\bullet$ be a multipolarisation on $X$.
If the twisted cotangent sheaf $\Omega_X \otimes A$ is not generically ample with respect to 
$H_\bullet$, one of the following holds:
\begin{enumerate}
\item There exists a normal projective variety $Y$ of dimension at most $n-1$
and a vector bundle $V$ on $Y$ such that
$X' := \PP(V)$
admits a birational morphism \holom{\mu}{X'}{X} such that $\sO_{X'}(\mu^* A) \simeq \sO_{\PP(V)}(1)$; or
\item $(X,  \sO_X(A)) \simeq (Q^n, \sO_{Q^n}(1))$ where $Q^n \subset \PP^{n+1}$ is a (not necessarily smooth) quadric hypersurface.
\end{enumerate}
\end{theorem}

\begin{proof}
By hypothesis the twisted cotangent sheaf is not generically ample
with respect to the multipolarisation $H_\bullet$, so the first piece $\sF_1 \subset T_X$ of the Harder-Narasimhan-filtration of $T_X$ satisfies
$\mu_{H_\bullet}(\sF_1 \otimes A^*) \geq 0$. Thus we are in the situation of Setup \ref{setup}, hence 
$\sF_1$ defines an algebraic foliation and the general fibres $(F, \sO_F(\mu^* A))$ of the graph $X' \rightarrow Y$ are classified in Lemma \ref{lemmaidentifyfibres}. If $(F, \sO_F(\mu^* A))$ is a linear space, the fibration
$X' \rightarrow Y$ is a projective bundle by Proposition \ref{propositionlinearspaces},a).
If $(F, \sO_F(\mu^* A))$ is a $\PP^{d-1}$-bundle or has a cone structure, we replace $X' \rightarrow Y$
by the projective bundle $\tilde X \rightarrow \tilde Y$ constructed in 
Proposition \ref{propositionlinearspaces},b).
If $(F, \sO_F(\mu^* A))$ is a quadric, then 
we know by Proposition \ref{propositionquadric} that $X$ itself is a quadric.
\end{proof}

So far all our considerations were valid for an arbitrary multipolarisation $H_\bullet$. 
However it is easy to see that the first part of Theorem \ref{theoremmain}
is not valid for an arbitrary multipolarisation (cf. \cite[Sect.1.B]{a11}).
The following lemma will turn out to be crucial for the proof of Theorem \ref{theoremmain}:

\begin{lemma} \label{lemmainequality}
Let $B$ be a projective manifold of dimension $m \geq 1$, and let $V$ be a nef vector bundle of rank $d+1 \geq 2$ on $B$.
Let $\holom{\pi}{\PP(V)}{B}$ be the projectivisation of $V$, and 
let $\zeta$ be the first Chern class of the tautological bundle $\sO_{\PP(V)}(1)$. 
Then we have
$$
(K_{\PP(V)/B}+d \zeta) \cdot \zeta^{m+d-1} \geq 0.
$$
\end{lemma}

\begin{remark*}
If $V$ is globally generated, the statement is quite straightforward: intersecting $d$ general elements of the 
free linear system $|\sO_{\PP(V)}(1)|$ we obtain a projective manifold $Z \subset X$ such that
the induced map $\pi|_Z: Z \rightarrow B$ is birational. By the adjunction formula we have
$$
(K_{\PP(V)/B}+d \zeta)|_Z \simeq K_{Z/B}
$$
which is an effective divisor since $B$ is smooth. Thus we have
$$
(K_{\PP(V)/B}+d \zeta) \cdot \zeta^{m+d-1} = K_{Z/B} \cdot (\zeta|_Z)^{m-1} \geq 0.
$$
\end{remark*}

\begin{proof}
By the canonical bundle formula we have $K_{\PP(V)/B}+d \zeta = \pi^* \det V - \zeta$. 
If $m=1$ the statement immediately follows from the equality $\zeta^{d+1}=\pi^* \det V \cdot \zeta^d$.

Suppose now that $m \geq 2$. For $k \in \N$ we denote by $s_{(1^k)}(V)$ the Schur polynomial of degree $k$ associated
to the partition $\lambda_i=1$ for $i=1, \ldots, k$ (cf. \cite[Ch.8.3]{Laz04} for the relevant definitions).  
By \cite[Ex.8.3.5]{Laz04} we have
$$
\pi^* \det V \cdot  \zeta^{m+d-1} = s_{(1)}(V) \cdot s_{(1^{m-1})}(V)
$$
and
$$
\zeta^{m+d} = s_{(1^{m})}(V).
$$
Yet by the Littlewood-Richardson rule \cite[Lemma 14.5.3]{Ful98} we have
$$
s_{(1)}(V) \cdot s_{(1^{m-1})}(V) = s_{(1^{m})}(V) + s_{(2, 1^{m-2}, 0)}(V),
$$
where $s_{(2, 1^{m-2}, 0)}(V)$ is the Schur polynomial of degree $m$ corresponding to the partition $\lambda_1=2$, $\lambda_m=0$ and $\lambda_i=1$
for all other $i$.
Thus we see that
$$
(K_{\PP(V)/B}+d \zeta) \cdot \zeta^{m+d-1} = s_{(2, 1^{m-2}, 0)}(V)
$$
which is non-negative by \cite[Thm.8.3.9]{Laz04}.
\end{proof}

\begin{proof}[Proof of Theorem \ref{theoremmain}]
Suppose that $(X,  \sO_X(A))$ is not isomorphic to  $(\PP^n, \sO_{\PP^n}(1))$.
Arguing by contradiction we suppose that $\Omega_X \otimes A$ 
is not generically nef with respect to the polarisation $A$.
Then the first piece $\sF_1 \subset T_X$ of the Harder-Narasimhan-filtration of $T_X$  satisfies
\begin{equation} \label{eqn1}
\mu_{A}(\sF_1 \otimes A^*) > 0.
\end{equation} 
Thus we are in the situation of Setup \ref{setup}, in particular 
$\sF_1$ defines an algebraic foliation of rank $r:=\rk \sF_1$. By Lemma \ref{lemmaidentifyfibres} the general fibre 
$(F, \sO_F(\mu^* A))$ of the graph $X' \rightarrow Y$ 
is a linear projective space. Thus we know by Proposition \ref{propositionlinearspaces},a)
that $X'$ is a projectivised bundle $\PP(V)$ where $V:= \varphi_* (\sO_{X'}(\mu^* A))$.
 
Let $\holom{\eta}{B}{Y}$ be a desingularisation, then we have $X' \times_Y B \simeq \PP(V_B)$
where $V_B:=\eta^* V$. Denote by $\holom{\nu}{\PP(V_B)}{X}$ the birational morphism obtained
by composing $\mu$ with the natural map $\PP(V_B) \rightarrow X'$.
By Lemma \ref{lemmainequality} we have
\begin{equation} \label{eqn2}
(K_{\PP(V_B)/B}+ r \nu^* A) \cdot (\nu^* A)^{n-1} \geq 0.
\end{equation}
The slope $\mu_A(\sF_1 \otimes A^*)$ is a positive multiple of the intersection product
$$
(-K_{\sF_1} - r A) \cdot A^{n-1}.
$$
Since $A$ is ample we can represent (a positive multiple of) $A^{n-1}$ by a MR-curve $C$ that does not meet
the image of $\nu$-exceptional locus. Thus there exists an open neighbourhood $C \subset U \subset X$ such that
$K_{\sF_1}|_U = K_{\PP(V_B)/B}|_U$. In particular we have
$$
(K_{\PP(V_B)/B}+ r \nu^* A) \cdot (\nu^* A)^{n-1} = (K_{\sF_1} + r A) \cdot A^{n-1}.
$$
Yet this shows that \eqref{eqn1} contradicts \eqref{eqn2}.
This shows the first part of the statement, the second part is a special case of Theorem \ref{theoremmultipolarisation}. 
\end{proof}

As mentioned in the introduction, the proof of Theorem \ref{theoremKO} is essentially a combination
of arguments and results due to Araujo and Druel \cite{AD12, AD13}. The new ingredient is
Theorem \ref{theoremmain} and some modifications due to our more general setting.

\begin{proof}[Proof of Theorem \ref{theoremKO}]
Let $\sG \subset \sF$ be a torsion-free subsheaf. 
By Theorem \ref{theoremmain} the twisted cotangent sheaf $\Omega_X \otimes A$ 
is generically nef with respect
to $A$, so we have
$$
\mu_A(\sG \otimes A^*) \leq 0,  
$$
and by hypothesis $\mu_A(\sF \otimes A^*)=0$. Thus $\sF$ is semistable with respect to the polarisation $A$ and 
$$
\det \sF \cdot A^{n-1}= r A^n>0.
$$
Thus the restriction $\sF|_C$ to a MR-general curve $C$ is semistable with ample determinant, hence
it is ample. In particular the Bogomolov-McQuillan theorem applies and the general $\sF$-leaf is algebraic 
with rationally connected closure. Let $Y$ be the normalisation of the closure in $\Chow{X}$
of the locus parametrising the closure of the general $\sF$-leaves, 
and let $X'$ be the normalisation of the universal family over $Y$. 
By construction the natural map \holom{\mu}{X'}{X} is birational and the fibration  $\holom{\varphi}{X'}{Y}$
is equidimensional of dimension $r$, 
the general fibre $F$ being the normalisation of the closure of a general $\sF$-leaf. 

{\em Step 1. Description of $X'$.}
As in the situation of 
Setup \ref{setup} we could now use Lemma \ref{lemmaidentifyfibres} to describe $F$, but the log-leaf structure of Araujo-Druel gives a more precise information: by \cite[Rem.3.11]{AD12} there exists an effective
Weil $\Q$-divisor $\Delta$ such that 
$$
K_{X'/Y} + \Delta \sim_\Q \mu^* K_{\sF} \sim_\Q -r \mu^* A.
$$
In particular $(F, \Delta \cap F)$ is a log Fano variety of dimension $r$ and index $r$. By \cite[Thm.2.5]{AD12}
this implies that $F$ is a projective space or a quadric. Moreover by \cite[Lemma 2.2]{AD12} the pair
$(F, \Delta \cap F)$ is klt unless $F \simeq \PP^r$ and $\Delta \cap F$ is a hyperplane.
Yet $(F, \Delta \cap F)$ is not klt, since $\det \sF$ is ample \cite[Prop.3.13]{AD12}. Thus
$(F, \sO_F(\mu^* A))$ is a linear projective space, in particular by
Proposition \ref{propositionlinearspaces},a) we have
$X' \simeq \PP(V)$ with $V:=\varphi_*(\sO_{X'}(\mu^* A))$. We claim that $\det V$ is ample:
by \cite[Prop.2]{Fuj92} the line bundle
$\det V$ is semiample, so it is sufficient
to prove that $\det V \cdot C>0$ for every curve $C \subset Y$. If this was not the case the restriction $V|_C$
would be a nef vector bundle with numerically trivial determinant, hence if we set $X'_C := \fibre{\varphi}{C}$
and $\holom{\varphi_C:=\varphi|_{X_C'}}{X_C'}{C}$,
then
$$
(\mu^* A)^{r+1} \cdot X'_C = \varphi_C^* (\det V|_C) \cdot (A|_{X_C'})^r = 0.
$$
Since $A$ is ample this implies that $X'_C$ is contracted by $\mu$ onto a subvariety of dimension at most $r$. In particular all the points of the curve $C$ parametrise the same cycle in $X$. This contradicts
the construction of $Y$ as a normalisation of a subvariety in the Chow variety $\Chow{X}$.

{\em Step 2. Description of $\Delta$.}
Note that $\Delta$ is $\Q$-Cartier, since $K_{X'/Y}$ and $\mu^* K_{\sF}$ are $\Q$-Cartier. We have also
seen that $\Delta$ has an irreducible component $E$ with coefficient one such that the general fibre
of $\varphi|_E: E \rightarrow Y$ is a projective space of dimension $r-1$. We claim that $\Delta=E$ and $E \rightarrow Y$
is a $\PP^{r-1}$-bundle. 

{\em Proof of the claim.}
If $E \rightarrow Y$ is not a $\PP^{r-1}$-bundle there exists a
$\varphi$-fibre $F_0=\fibre{\varphi}{0}$ contained in $\Delta$.  The same holds if the support of $\Delta$ is reducible since $E$ is the unique component of $\Delta$ that surjects onto $Y$. Arguing by contradiction we suppose that
there exists a $\varphi$-fibre $F_0$ contained in $\Delta$. Let $C \rightarrow Y$ be a general non-constant morphism
such that $0 \in C$, and set $X'_C:= X' \times_Y C$.
Denote by $\holom{\nu}{X'_C}{X'}$ the natural map to the first factor, and by \holom{\varphi_C}{X_C'}{C} the $\PP^r$-bundle structure.
Note that we have $X_C' \simeq \PP(V_C)$ where $V_C:= (\varphi_C)_* (\sO_{X_C'}(\nu^* \mu^* A))$, 
moreover by construction
$$
\nu^* \Delta = \Delta' + F_0
$$
with $\Delta'$ an effective $\Q$-divisor. Since we have
$$
(K_{X_C'/C}+\nu^* \Delta) \sim_\Q \nu^* \mu^* K_{\sF} \sim_\Q - r \nu^* \mu^* A,
$$
and $K_{X_C'/C} = \varphi_C^* \det V_C - (r+1) \nu^* \mu^* A$, we obtain
$$
\Delta' \sim_\Q \nu^* \mu^* A - \varphi_C^* \det V_C - F_0.
$$
However by \cite[Lemma 4.12,b)]{AD12} no multiple of $\nu^* \mu^* A - \varphi_C^* \det V_C - F_0$
has a global section, a contradiction.

{\em Step 3. The $\mu$-exceptional locus.} We first claim that the $\mu$-exceptional locus
is contained in the divisor $E=\Delta$: let $C$ be a curve in $X'$ such that $\mu(C)$ is a point.
Then we have $\mu^* A \cdot C=0$, so $E \cdot C = - K_{X'/Y} \cdot C$. Since $\mu$ is finite on the $\varphi$-fibres and $\det V$ is ample we have $\varphi^* \det V \cdot C>0$.
Therefore we have
$$
E \cdot C = - K_{X'/Y} \cdot C = (\varphi^* \det V^* + (r+1) \mu^* A) \cdot C < 0,
$$
hence $C$ is contained in $E$. 
Set now $W:= (\varphi|_E)_* (\sO_E(\mu^* A))$, then $W$ is a nef vector bundle of rank $r$.
We have already seen that
$$
E \in |
\sO_{\PP(V)}(1) \otimes \varphi^* \det V^*
|,
$$
so pushing down the exact sequence
$$
0 \rightarrow \sO_{X'}(-E+\mu^* A) \rightarrow \sO_{X'}(\mu^* A) \rightarrow \sO_E(\mu^* A) \rightarrow 0
$$
to $Y$ we obtain an exact sequence
\begin{equation} \label{eqnextension}
0 \rightarrow \det V \rightarrow V \rightarrow W \rightarrow 0
\end{equation}
and $\det W \simeq \sO_W$. Thus $W$ is a nef vector bundle with trivial determinant, 
moreover we have a morphism $\holom{\mu_E}{E}{X}$ such that $\sO_{\PP(W)}(1) \simeq \sO_E(\mu_E^* A)$.
By Lemma \ref{lemmatrivialbundle} below this implies that $W \simeq \sO_Y^{\oplus d}$. 
Since $\sO_{\PP(V)}(1) \simeq \sO_{X'}(\mu^* A)$ is semiample we know by
\cite[Cor.4]{Fuj92} that the exact sequence \eqref{eqnextension} splits,
thus $X$ is a generalised cone in the sense of \cite[1.1.8]{BS95}. 
\end{proof}

\begin{lemma} \label{lemmatrivialbundle}
Let $Y$ be a normal projective variety, and let $W$ be a nef vector bundle of rank $r$ on $Y$ such that $\det W \equiv 0$.
Set $E:=\PP(W)$ and suppose that there exists a (not necessarily surjective) morphism $\holom{\mu_E}{E}{X}$ 
to a normal projective variety $X$, and an ample Cartier divisor $A$ on $X$ such that $\sO_{\PP(W)}(1) \simeq \sO_E(\mu_E^* A)$. 
Then we have $W \simeq \sO_Y^{\oplus r}$, in particular $E \simeq Y \times \PP^{r-1}$.
\end{lemma}

\begin{proof} Using the projection formula we see that we can suppose without loss of generality that
$Y$ is smooth. Since $W$ is nef with numerically trivial determinant, it is numerically flat. In particular
all the Chern classes $c_i(W)$ vanish \cite[Cor.1.19]{DPS94}, so by the usual relations for the
tautological divisor \cite[App.A, Sect.3]{Har77}
we see that 
$$
\sO_{\PP(W)}(1)^r \equiv 0 \ \mbox{in} \ H^{2r}(E, \R),
$$
i.e. the numerical dimension of $\sO_{\PP(W)}(1)$ is $r-1$. 
By hypothesis $\sO_{\PP(W)}(1) \simeq \sO_E(\mu_E^* A)$ is semiample, so some positive multiple induces a fibration
$\holom{\tau}{E}{G}$ onto some normal projective variety $G$ of dimension $r-1$. By the rigidity lemma one sees 
easily that $\mu_E$ factors through $\tau$, in particular there exists an ample Cartier divisor $A_G$ on $G$ such
that $\sO_{\PP(W)}(1) \simeq \sO_E(\tau^* A_G)$. Any fibre of the natural map $\PP(W) \rightarrow Y$ is a $\PP^{r-1}$
mapping surjectively onto $G$. Since we have
$$
1 = \sO_{\PP(W)}(1)^{r-1} \cdot \PP^{r-1} = \deg( \PP^{r-1} \rightarrow G ) \cdot A_G^{r-1} \geq  1,
$$
we see that $G \simeq \PP^{r-1}$ and $\sO_G(A_G) \simeq \sO_{\PP^{r-1}}(1)$. In particular we obtain
$$
h^0(Y, W) = h^0(E, \sO_{\PP(W)}(1)) = h^0(\PP^{r-1}, \sO_{\PP^{r-1}}(1)) = r.
$$
Since $W$ is numerically flat of rank $r$, this immediately implies that $W$ is trivial. 
\end{proof}

\def\cprime{$'$}

\end{document}